\documentclass[12pt,reqno]{amsart}
\usepackage{amssymb,amsmath}
\def\({\left(}
\def\){\right)}
\def\Nx{\nabla}

\def\eb{\varepsilon}
\def\al{\alpha}
\def\Om{\Omega}
\def\di{\partial_i}
\def\la{\lambda}

\def\R {\mathbb{R}}

\def \ta {\theta}

\def \and{\qquad\text{and}\qquad}

\def\Dx{\Delta}

\newtheorem{proposition}{Proposition}[section]
\newtheorem{theorem}[proposition]{Theorem}

\newtheorem{lemma}[proposition]{Lemma}
\theoremstyle{definition}
\newtheorem{definition}[proposition]{Definition}
\newtheorem{remark}[proposition]{Remark}

\numberwithin{equation}{section}

\def \au {\rm}

\def \no#1#2#3 {{\bf #1} (#3), #2.}
\def \eds#1#2#3 {#1, #2, #3.}

\title[Global Dynamics for 3D Navier-Stokes-Voight Equations]
{Global Attractors and  Determining Modes for the \\3D
Navier-Stokes-Voight Equations}

\author[]
{Varga K. Kalantarov and Edriss S. Titi}
\date{May 27, 2007}
\thanks{ }
\address{(V.K.Kalantarov) Department of mathematics, Ko{\c c} University,
\newline\indent Rumelifeneri Yolu, Sariyer 34450\newline\indent
Sariyer, Istanbul, Turkey} \email{vkalantarov@ku.edu.tr}

\address{(E.S.Titi) Department of mathematics and Department
\newline \indent of Mechanical and Aerospace
Engineering,\newline \indent University of California
\newline \indent Irvine, California 92697,USA. \newline \indent
 {\bf Also:} \newline \indent
 Department of Computer Science
and  Applied Mathematics \newline \indent Weizmann Institute of
Science \newline \indent Rehovot, 76100, Israel}
\email{etiti@math.uci.edu} \,\, \email{edriss.titi@weizmann.ac.il}

\subjclass[2000]{} \keywords{}

\begin{document}

\begin{abstract}
We investigate the long-term dynamics of the three-dimensional
Navier-Stokes-Voight model of viscoelastic incompressible fluid.
Specifically, we derive upper bounds for the number of determining
modes for the 3D Navier-Stokes-Voight equations and for the
dimension of a global attractor of a semigroup generated by these
equations. Viewed from the numerical analysis point of view we
consider the Navier-Stokes-Voight model as a non-viscous (inviscid)
regularization of the three-dimensional Navier-Stokes equations.
Furthermore, we also show that the weak solutions of the Navier-
Stokes-Voight equations converge, in  the appropriate norm, to the
weak solutions of the inviscid simplified Bardina model, as the
viscosity coefficient $\nu \rightarrow 0$.
\end{abstract}

\maketitle



\section{ Introduction }

 \noindent We consider the three-dimensional Navier-Stokes-Voight (NSV) system of
 equations
\begin{equation}\label{a1}
 v_{t}-\nu \Dx v-\al^2 \Dx v_t+(v\cdot \Nx)v+\Nx p=f(x), \ x \in \Om, t
 \in \R^+,
\end{equation}
\begin{equation}\label{a2}
\mbox{div}\ v=0 , \ x \in  \Om, t \in \R^+; v(x,t)=0, \ x \in
\partial \Om, t\in \R^+,
\end{equation}
\begin{equation}\label{a3}
v(x,0)=v_0(x),\ x \in \Om ,
\end{equation}
where $\Om\subset \R^3$ is a bounded domain with sufficiently smooth
boundary $\partial \Om$, $v=v(x,t)$ is the velocity vector field,
$p$ is the pressure, $\nu>0$  is the kinematic viscosity, $\al$ is a
length scale parameter characterizing the elasticity of the fluid,
and $f$ is a given force field.\\

The system (\ref{a1})-(\ref{a2}) models the dynamics of a
Kelvin-Voight viscoelastic incompressible fluid and was introduced
by A.P. Oskolkov in \cite{Osk1} as a model of motion of linear,
viscoelastic fluids.\\

 The viscous simplified Bardina model was introduced and
 studied in \cite{L1} (see also \cite{L2})
 as a simplified version of the Bardina sub-grid scale model of
 turbulence \cite{BFR}. In \cite{CLT} the viscous and inviscid
 simplified Bardina model were shown to be globally well-posed. It
 is interesting to observe that the inviscid
 simplified Bardina model coincides with the inviscid version of the
 NSV equations \eqref{a1}-\eqref{a3}. Viewed from the numerical
 analysis point of view the authors of \cite{CLT} proposed the
 inviscid simplified Bardina model (or equivalently the inviscid
 NSV equations) as a non-viscous (inviscid) regularization of the 3D Euler
 equations, subject to periodic boundary conditions. Motivated by
 this observation the system \eqref{a1}-\eqref{a3} was also
 proposed in \cite{CLT} as a regularization, for small values of
 $\alpha$, of the 3D Navier-Stokes (NS) equations for the purpose
 of direct numerical simulations for both the periodic and the no-slip
 Dirichlet boundary conditions. \\

In \cite{Osk1} it is shown that the initial boundary value problem
(\ref{a1})-(\ref{a3}) has a unique weak solution. In \cite{Ka1} and
\cite{Ka2} it is shown that the semigroup generated by the problem
(\ref{a1})-(\ref{a3}) has
 a finite dimensional global attractor. \\

 In this paper we give an estimate of the fractal and
 Hausdorff dimensions of the
 global attractor of a dynamical system generated by the problem
 \eqref{a1}-\eqref{a3}, which is an improvement of the estimates
 done in \cite{Ka2}. Moreover, we derive estimates for the number of asymptotic
 determining modes of the solutions of the problem \eqref{a1}-\eqref{a3}.
We also show that there exists a number $m$ such that each
trajectory $v(t)$ on the global attractor of the dynamical system
generated by this problem is uniquely determined by its projection
$P_mv(t)$ onto the $\mbox{span}\{w_1,...,w_m\}$  of the first $m$
eigenfunctions of the Stokes operator. This observation is related
to the notion of continuous data assimilations as it has been
presented in \cite{KrYs},\cite{OlTi} and \cite{OlTi2}. \\

It is worth stressing that by adding the regularizing term
$(-\alpha\Delta v_t)$ to the NS equations the system
(\ref{a1})-(\ref{a3}) changes its parabolic character. In
particular, the 3D system (\ref{a1})-(\ref{a3}) is globally
well-posed forward and backwards in time. The semigroup generated by
the problem (\ref{a1})-(\ref{a3}) is only asymptotically compact. In
this sense the system is similar to damped hyperbolic systems. We
also remark that this type of inviscid regularization has been
recently used for the two-dimensional surface quasi-geostropic model
\cite{BT}. In particular, necessary and sufficient conditions for
the formation of singularity were presented in terms of regularizing
parameter.

\section{Preliminary}

\noindent In this paper we will be using the following standard
notations in the mathematical theory of NS equations:

\begin{itemize}

\item $L^p(\Om), 1\leq p \leq\infty$, and $H^s(\Om)$ are the usual Lebesgue and Sobolev spaces, respectively.

\item For $v =(v_1,v_2,v_3),$ and $ u=(u_1,u_2,u_3)$ we denote by
$$(u,v)=\sum\limits_{j=1}^3(v_j,u_j)_{L^2(\Om)}, \
\|v\|^2=\sum\limits_{j=1}^3\|v_i\|_{L^2(\Om)}^2, \ \|\Nx v\|^2:=
\sum\limits_{j,i=1}^3\|\di v_j\|^2_{L^2(\Om)}.$$

\item We set
$$
\mathcal{V}:=\left\{v \in (C_0^{\infty}(\Om))^3: \ \Nx\cdot
v=0\right\}.
$$

\item
$H$  is the closure  of the set   $\mathcal{V}$ in $(L_2(\Omega))^3$
topology.

\item P is the Helmholz-Leray orthogonal projection in
$(L^2(\Omega))^3$ onto the space $H$, and $h:=Pf.$
\item $A:=-P\Dx$ is the Stokes operator subject to the no-slip homogeneous
Dirichlet boundary condition with the domain $(H^2(\Om))^3\cap V$.
The operator $A$ is a self-adjoint positively definite operator in
$H$, whose inverse $A^{-1}$ is a compact operator from $H$ into $H$.
Thus it has an orthonormal system of eigenfunctions
$\{w_j\}_{j=1}^{\infty}$  of $A.$

\item We denote by  $\{\la_j\}_{j=1}^{\infty}, 0<\la_1 \leq \la_2\leq \cdots,$
the eigenvalues of the Stokes operator $A$ corresponding to
eigenfunctions $\{w_j\}_{j=1}^{\infty}$, repeated according to their
multiplicities.

\item $V_s:=D(A^{s/2}), \|v\|_s:=\|A^{s/2}v\|, s \in \R. \  V:=V_1= (H_0^1(\Om))^3\cap H$
 is the Hilbert space with
the norm $\|v\|_1=\|u\|_V=\|\Nx u\|$, thanks to the Poincar\'{e}
inequality  \eqref{Poi}. Clearly $V_0=H$.
\item  For $u,v,w \in \mathcal{V}$ we define the following bilinear
form
$$
B(u,v):=P\left((u\cdot\Nx)v\right)\ \mbox{and the trilinear form} \
b(u,v,w)=(B(u,v),w).
$$

\noindent The bilinear form $B(\cdot,\cdot)$ can be extended as a
continuous operator $B: V\times V \rightarrow V'$, where $V'$ is the
dual of $V$ (see, e.g., \cite{CF}).

\item  For each $u,v,w \in V$
\begin{equation}\label{bfc}
b(u,v,v)=0,  \ \mbox{and} \ b(u,v,w)= -b(u,w,v).
\end{equation}
\end{itemize}

Next we formulate some well known inequalities and a Gronwall type
lemma that we will be using in what follows.\\

\noindent{\it Young's inequality}
\begin{equation}\label{You}
ab\leq \frac{\eb}p a^p+\frac1{q\eb^{1/(p-1)}}b^q, \ \mbox{for all} \
a,b,\eb>0, \ \mbox{with} \ q=p/(p-1), 1<p<\infty.
\end{equation}

\noindent{\it Poincar\'{e} inequality}

\begin{equation}\label{Poi}
\|u\|\leq \la_1^{-1/2}\|u\|_1, \ \ \forall u \in V,
\end{equation}
where $\la_1$ is the first eigenvalue of the Stokes operator under
the homogeneous Dirichlet boundary condition.

\noindent Hereafter, $C$ will denote a dimensionless scale invariant
constant which might depend on the shape of the domain $\Om$.\\

\noindent {\it Ladyzhenskaya inequalities}
(\cite{CF},\cite{La1},\cite{LSU})

\begin{equation}\label{int3}
\|u\|_{L^3}\leq C\|u\|^{1/2}\|\Nx u\|^{1/2}, \  \forall u \in V,
\end{equation}

\begin{equation}\label{Lad4}
\|u\|_{L^4}\leq C\|u\|^{1/4}\|u\|_1^{3/4},\  \forall u \in V.
\end{equation}

\noindent {\it Sobolev inequality} (see, e.g., \cite{Ad})

\begin{equation}\label{Sob}
\|u\|_{L^6}\leq C\|u\|_1, \  \forall u \in V.
\end{equation}

\noindent{\it Gagliardo - Nirenberg inequalities} (see, e.g.,
\cite{BV},\cite{CF},\cite{LSU})

\begin{equation}\label{GN1}
\|u\|_{L^{6/(3-2\eb)}} \leq C\|u\|^{1-\eb}\|u\|_{1}^{\eb}, \ \ 0\leq
\eb\leq 1 , \ \ \forall u \in V.
\end{equation}

\begin{equation}\label{GN}
\|u\|_{L^p} \leq C\|u\|^{2/p}\|u\|_{3/2}^{1-2/p}, \ \ p \in [2,
\infty), \ \ \forall u \in V_{3/2}.
\end{equation}

 \noindent {\it Agmon inequality} (see, e.g., \cite{CF})

\begin{equation}\label{Agm}
\|u\|_{L^{\infty}(\Om)}\leq C \|u\|_1^{1/2}\|A u\|^{1/2}, \ \forall
u \in V_2.
\end{equation}

\noindent We will use also the following estimates of   the
trilinear form $b(u,v,w)$ which follow from \eqref{int3} -
\eqref{Agm} (see, e.g., \cite{CF}).

\begin{equation}\label{bf4}
|b(u,v,w)|\leq C\|u\|^{1/2}\|u\|_1^{1/2}\|v\|_1\|w\|_1, \ \ \forall
u,v,w \in V,
\end{equation}

\begin{equation}\label{bf1}
|b(u,v,u)|\leq C\|u\|^{1/2}\|u\|_1^{3/2}\|v\|_1, \ \ \forall u,v \in
V,
\end{equation}

\begin{equation}\label{bf1ad}
|b(u,v,w)|\leq C \|u\|_1\|v\|_1\|w\|^{1/2}\|w\|_1^{1/2}, \ \ \forall
u,v,w \in V,
\end{equation}

\begin{equation}\label{bf2}
|b(u,v,w)|\leq C\la_1^{1/4}\|u\|_1\|v\|_1\|w\|_1, \ \ \forall u,v,w
\in V.
\end{equation}

\begin{lemma}\label{le1} (\cite{JT}, see also \cite{FMRT})\ Let $a(t)$ and $b(t)$ be
locally integrable functions on $(0,\infty)$ which satisfy for some
$T>0$ the conditions
$$
\liminf_{t\rightarrow \infty}\frac
1T\int_{t}^{t+T}a(\tau)d\tau=\gamma,\ \limsup_{t\rightarrow
\infty}\frac 1T\int_{t}^{t+T}a^{-}(\tau)d\tau=\Gamma,\
\liminf_{t\rightarrow \infty}\frac
1T\int_{t}^{t+T}b^{+}(\tau)d\tau=0,
$$
where $\gamma >0, \Gamma < \infty,$ $a^{-}=\max\{-a,0\}$ and
$b^{+}=\max\{b,0\}.$ If a non-negative, absolutely continuous
function $\phi(t)$, satisfies
$$
\phi'(t)+a(t)\phi(t)\leq b(t), \ t \in (0,\infty),
$$
then $\phi(t)\rightarrow 0 $ as $t\rightarrow \infty$.\\
\end{lemma}

\begin{definition}\label{ak}(see, e.g., \cite{FMRT}, \cite{He},\cite{La2}) A
semigroup $S(t): V \rightarrow V, t\geq 0$ is called {\it
asymptotically compact}, if for any sequence of positive numbers
$t_n\rightarrow \infty$ and any bounded sequence $\{v_n\}\subset V$
the sequence $\{S(t_n)v_n\}$ is precompact in $V$.
\end{definition}
\begin{theorem}\label{AK} (see, e.g., \cite{He},\cite{La2},\cite{Tem}) Assume that a semigroup
 $S(t): V \rightarrow V,$ for $t\geq
t_0>0$ can be decomposed into the form
$$
S(t)= Y(t) + Z(t),
$$
where $Z(t)$ is a compact operator in $V$ for each $t\geq t_0>0$.
Assume also that there is a continuous function $k:
[t_0,\infty)\times \R^+ \rightarrow \R^+$ such that for every $R>0$
$k(t,R)\rightarrow 0$  as $t\rightarrow \infty$ and
$$
\|Y(t)v \|_V\leq k(t,R), \ \ \ \mbox{for all} \ \ t\geq t_0>0, \ \ \
\mbox{and all} \ \  \|v\|_V\leq R.
$$

\noindent Then $S(t): V \rightarrow V, t\geq 0$ is asymptotically
compact.
\end{theorem}

\noindent Next we state a result from \cite{La2} which will enable
us to estimate the dimension of the global attractor for the system
(\ref{a1})-(\ref{a3}). This result is typically useful in the
context of nonlinear damped hyperbolic systems, when the damping
term is not strong enough to control the instabilities rising from
the perturbed nonlinearity.

\begin{theorem}\label{frdim} (see \cite{CFNT}, \cite{La2})
Let $S(t),t\in\R^+$, be a semigroup generated by the problem
$$
v_t(t)=\Phi(v(t)), \ \ v\big|_{t=0}=v_0,
$$
in the phase space $H$ and let $\mathcal{M}\subset H$ is a compact
invariant subset with respect to $S(t)$. Let $S(t)$ and
$\Phi(\cdot)$ be uniformly differentiable on $\mathcal{M}$ and let
$L(t,v_0)$ be a differential of $\Phi$ at the point $S(t)v_0, v_0\in
\mathcal{M}$. Suppose that $L^c(t,v_0):=L(t,v_0)+L^*(t,v_0), v_0\in
\mathcal{M}$ satisfies the inequality
\begin{equation}
(L^c(t)u,u)\leq -h_0(t)\|u\|^2+\sum_{k=1}^m h_{s_k}(t)\|u\|_{s_k}^2,
\end{equation}
for some numbers $ s_k <0, (k=1,...,m)$ and some functions
$h_0,h_{s_k}\in L_{1,loc}(\R), h_{s_k}(t)\geq 0, h_0(t)\geq 0
\ \mbox{for all} \ t \in \R^+.$\\
\noindent Then
$$
\mbox{dim}_{\mathcal{H}}(\mathcal{M})\leq
\mbox{dim}_{f}(\mathcal{M})\leq N,
$$
where $N$ is such that
$$
-\bar{h}_0(T)+\sum_{k=0}^m\bar{h}_{s_k}(T)N^{s_k}<0,
$$
for some $T>0$. Here $\bar{h}_i(T):=\frac1T\int_0^Th_i(\tau)d\tau.$
\end{theorem}

 \section{ Existence of Global Attractors}

{\noindent}Applying the Helmholtz - Leray projector $P$ to the
system (\ref{a1})-(\ref{a2}), we obtain the following  equivalent
functional differential equation
\begin{equation}\label{a4}
v_t +\nu Av+\al^2 Av_t+B(v,v)=h,\ \ \  h=Pf,
\end{equation}
\begin{equation}\label{a5}
v(0)=v_0.
\end{equation}

\noindent The question of global existence and uniqueness of
\eqref{a4}-\eqref{a5} first was studied  in \cite{Osk1}, where
actually  it was established that the problem (\ref{a1})-(\ref{a3})
generates a continuous semigroup $S(t): V\rightarrow V, t\in \R^+$.
In \cite{CLT} the authors proved also the global regularity for
inviscid model
of \eqref{a4}, i.e. when $\nu=0.$\\

In this section we show that the semigroup $S(t)$ generated by the
problem (\ref{a1})-(\ref{a3}) has an absorbing ball in $V$ and an
absorbing ball in $V_2$. Then we show that $S(t): V\rightarrow V,$
for $t\in \R^+$ is an asymptotically compact semigroup,
and deduce the existence of a global attractor in $V$.\\

Let us note that the formal estimates  we provide below
 can be justified rigorously by using a Galerkin approximation
procedure and passing to the limit, by using the relevant Aubin's
compactness theorem as for the NS
equations ( see, for example,
\cite{CF},\cite{FMRT},  \cite{Ro} or \cite{Tem}).\\

\noindent {\bf Absorbing ball in $V$.} Taking the inner product of
(\ref{a4}) with $v$, and noting that due to \eqref{bfc}
$(B(v,v),v)=0$, we get
\begin{equation}\label{a6}
\frac{d}{dt}\left[\|v(t)\|^2+ \al^2\|v(t)\|_1^2\right]+2\nu
\|v(t)\|_1^2\leq 2\|h\|_{-1}\|v(t)\|_1.
\end{equation}

\noindent It is easy to see by Poincar\'{e} inequality \eqref{Poi}
that
$$
\nu  \|v(t)\|_1^2\geq \frac{\nu}2\left[\la_1\|v\|^2+
\|v(t)\|_1^2\right] \geq d_0\left[\|v(t)\|^2+\al^2
\|v(t)\|_1^2\right],
$$
where $d_0:=\frac{\nu}2 \min \{\frac 1{\al^2},\la_1\}=\nu d_1$.
Hence (\ref{a6}) implies
$$
\frac{d}{dt}\left[\|v(t)\|^2+\al^2
\|v(t)\|_1^2\right]+d_0\left[\|v(t)\|^2+\al^2
\|v(t)\|_1^2\right]\leq \frac1{\nu}\|h\|_{-1}^2.
$$
By Gronwall's inequality we have
$$
\|v(t)\|^2+\al^2\|v(t)\|_1^2\leq
$$
$$
e^{-d_0(t-s)}\left[\|v(s)\|^2+\al^2\|v(s)\|_1^2-\frac{\|h\|^2_{-1}}{\nu
d_0}\right]+\frac1{\nu d_0}\|h\|_{-1}^2. \eqno(E_2)
$$
Therefore,
$$
\limsup_{t\rightarrow \infty}\left[\|v(t)\|^2+\al^2
\|v(t)\|_1^2\right]\leq \frac{\|h\|^2_{-1}}{\nu d_0}.\eqno(E_1)
$$

\noindent The last inequality implies that the semigroup $ S(t):
V\rightarrow V, t\in \R^+$ generated by the problem
(\ref{a1})-(\ref{a3}) (or equivalently \eqref{a4}-\eqref{a5}) has
an absorbing ball

\begin{equation}\label{ball1}
\mathcal{B}_1:=\left\{v\in V: \|v\|_1\leq \frac{2}{\sqrt{\nu \al^2
d_0}}\|h\|_{-1}\right\}.
\end{equation}

\noindent Hence, the following uniform estimate is valid
\begin{equation}\label{a7}
\| v(t)\|_1\leq M_1,
\end{equation}
where $ M_1=\frac{2}{\nu\al\sqrt{
d_1}}\|h\|_{-1}$, for $t$ large enough ( $t\gg 1$)
depending on the initial data.\\

\noindent {\bf Asymptotic compactness.} By using the Galerkin
procedure it is not difficult to prove the following

\begin{proposition}\label{le2} Let $s\in \R$. If $w_0\in V_s, g \in L^2([0,T); V_{s-2})$ then the linear problem
\begin{equation}\label{a71}
z_t+\al^2 A z_t+\nu Az= g(t) ,\ z(0)= 0
\end{equation}
has a unique weak solution which belongs to $C([0,T);V_s)$
 and the following inequality holds
$$
 \sup\limits_{t\in[0,T)}\|z(t)\|_s\leq C \|g\|_{L^{2}(0,T;V_{s-2})}, \  s
 \in \R.
$$
\end{proposition}

\begin{proposition}\label{prop1}  Let $h\in H$, be
time independent, then the
semigroup $S(t), t\geq 0$ is asymptotically
compact semigroup in $V$.\\
\end{proposition}

\begin{proof} Let $v_0\in V$. First we observe that  $S(t)$ has the representation
\begin{equation}\label{rep}
S(t)v_0=Y(t)v_0+Z(t)v_0,
\end{equation}
where $Y(t)$ is the semigroup, generated by the linear problem
\begin{equation}\label{a72}
y_t+\nu Ay+\al^2 A y_t=0 ,\ y(0)=v_0,
\end{equation}
and $z(t)=Z(t)(v_0)$ is the solution of the problem
\begin{equation}\label{a73}
z_t+\nu Az+\al^2 A z_t= h -B(v(t),v(t)) ,\ z(0)=0,
\end{equation}
where $v$ is the solution of (\ref{a1})-(\ref{a3}) (or
equivalently \eqref{a4}-\eqref{a5}) with the initial data $v_0$.\\
\noindent Taking the $H$ inner product of (\ref{a72}) with $y$  we
obtain
$$
\frac{d}{dt}\left[\|y(t)\|^2+\al^2
\|y(t)\|_1^2\right]+d_0\left[\|y(t)\|^2+\al^2
\|y(t)\|_1^2\right]\leq 0,
$$
where  we recall that $d_0=\nu d_1=\nu \frac{1}2\min
\{\frac1{\al^2},\la_1\}.$

\noindent This inequality implies that
\begin{equation}\label{con}
\|y(t)\|^2+\al^2 \|y(t)\|_1^2\leq e^{-d_0t}\left[\|v_0\|^2+\al^2
\|v_0\|_1^2\right], \ \mbox{for all} \  t>0.
\end{equation}
So the semigroup $Y(t):V\rightarrow V$ is exponentially
contractive.\\

\noindent Due to H\"{o}lder's inequality and the Sobolev inequality
\eqref{Sob} we have
\begin{multline*}
\|B(v,v)\|_{-1/2}=\sup_{\phi\in
V,\|A^{1/4}\phi\|=1}b(v,v,\phi)=\\\sup_{\phi\in
V,\|A^{1/4}\phi\|=1}\int_{\Om}P((v\cdot\Nx)v)\cdot\phi dx=
\sup_{\phi\in V,\|A^{1/4}\phi\|=1}\int_{\Om}(v\cdot\Nx)v\cdot
P\phi
dx= \\
\sup_{\phi\in V,\|A^{1/4}\phi\|=1}\int_{\Om}(v\cdot\Nx)v\cdot \phi
dx\leq C \sup_{\phi\in
V,\|A^{1/4}\phi\|=1}\|v\|_{L^6}\|v\|_1\|\phi\|_{L^3}.
\end{multline*}
Hence due to the Sobolev inequality $\|\phi\|_{L^3}\leq
C\|A^{1/4}\phi\|$ and \eqref{Sob} we have
\begin{equation}\label{Bvv1}
\|B(v,v)\|_{-1/2}\leq C\sup_{\phi\in
V,\|A^{1/4}\phi\|=1}\|v\|_1^2\|A^{1/4}\phi\|\leq C\|v\|_1^2,
\end{equation}
and
$$
B(v,v) \in L^{\infty}(\R^+;V_{-1/2}).
$$

\noindent The function  $v(t)$ as a solution of the problem
\eqref{a4}-\eqref{a5} with $v_0\in V$  belongs to
$L^{\infty}(\R^+;V)$. Thus due  to the inequality \eqref{Bvv1} and
the Proposition \ref{le2}, the solution of the problem (\ref{a73})
belongs to $ C(\R^+; V_{3/2})$, that is the operator $Z(t)$ maps $V$
into $V_{3/2}$. Since the embedding $V_{3/2}\subset V$ is a compact
embedding, the operator $Z(t)$ is a compact operator for each $t>0.$
Hence, the semigroup $S(t)$ satisfies the conditions of the Theorem
\ref{AK}, and is an asymptotically compact semigroup.
\end{proof}

\noindent Since each bounded dissipative and asymptotically compact
semigroup possesses a compact global attractor (see, e.g.,
\cite{BV}, \cite{He}, \cite{La1}, \cite{Tem}) we have:

\begin{theorem}\label{atr1} If $h \in H$ then the semigroup $S(t): V\rightarrow V$
has an absorbing ball $\mathcal{B}_1=\{ v \in V: \|v\|_1 \leq M_1\}$
and a global attractor $\mathcal{A}_1\subset V$. The attractor
$\mathcal{A}_1 $ is compact, connected and invariant.
\end{theorem}

\noindent Next we  show that the global attractor $\mathcal{A}_1$ is
a bounded subset of $V_2$.\\

\noindent Taking the inner product in $V_{1/2}$ of the equation
\eqref{a73} with $z$, and remembering that $v(t)=y(t)+z(t)\in
\mathcal{A}_1$, we get
\begin{multline}\label{dmm1}
\frac
 d{dt}\left[\|z(t)\|^2_{1/2}+\al^2\|z(t)\|^2_{3/2}\right]+
 2\nu \|z(t)\|^2_{3/2}=\\2(h,z(t))_{1/2}-2(B(v(t),v(t)),z(t))_{1/2}.
\end{multline}

\noindent The first term on the right-hand side has the estimate
$$
|2(h,z(t))_{1/2}|\leq 2\|h\|_{-1/2}\|z(t)\|_{3/2}\leq
\frac{\nu}2\|z(t)\|^2_{3/2}+\frac2{\nu}\|h\|^2_{-1/2}.
$$
The second term, due to \eqref{Bvv1}, has the following estimate
\begin{multline*}
|2(B(v(t),v(t)),z(t))_{1/2}|\leq
C\|(B(v(t),v(t))\|_{-1/2}\|z(t)\|_{3/2}\leq\\
\frac{\nu}2\|z(t)\|^2_{3/2}+ \frac
{C}{\nu}\|B(v(t),v(t))\|_{-1/2}^2\leq
\frac{\nu}2\|z(t)\|^2_{3/2}+\frac {C}{\nu}\|v\|_1^4.
\end{multline*}
Taking into account the last two inequalities in \eqref{dmm1} we
obtain
\begin{multline*}
\frac d{dt}\left[\|z(t)\|^2_{1/2}+\al^2\|z(t)\|^2_{3/2}\right]+
2d_0\left[\|z(t)\|^2_{1/2}+\al^2\|z(t)\|^2_{3/2}\right]\leq
\\
\frac{C}{\nu}\left(\|v(t)\|_1^{4}+ \|h\|^2_{-1/2}\right).
\end{multline*}
Integrating the last inequality we obtain the estimate
\begin{equation}\label{dmm2}
\|z(t)\|^2_{3/2}\leq \frac C{d_0\al^2\nu}\left(
M_1^{4}+\|h\|^2_{-1/2}\right)=L_0.
\end{equation}
Since the attractor $\mathcal{A}_1$ is invariant,
$S(t)\mathcal{A}_1=\mathcal{A}_1$, and due to \eqref{con} the
inequality
$$
\|v(t)-z(t)\|_1=\|y(t)\|_1\leq C(\|y(0)\|_1)e^{-d_0 t}
$$
holds, we deduce that for each  $u \in \mathcal{A}_1$  there exists
a sequence $\{z(t_k)\}, t_k\rightarrow \infty,$ corresponding to
$v_k(0) \in \mathcal{A}_1$, such that
\begin{equation}\label{lim}
u=\lim_{k\rightarrow \infty}z(t_k), \ v_k(0) \in \mathcal{A}_1.
\end{equation}
Thanks to \eqref{dmm2} the sequence $\{z(t_k)\}$ is belonging to a
ball in $V_{3/2}$, whose radius $L_0$ depends only on $M_1$ and
$\|h\|$. Hence, the sequence $\{z(t_k)\}$ is weakly compact in
$V_{3/2}$. Thus, by using \eqref{lim} and the inequality
$\|u\|_{3/2}\leq \lim\inf_{t_k\rightarrow \infty}\|z(t_k)\|_{3/2},$
we see
that $\mathcal{A}_1$ is bounded in $V_{3/2}$.\\
Knowing that $\mathcal{A}_1$ is bounded in $V_{3/2}$ we can use
similar arguments to show that $\mathcal{A}_1$ is also bounded in
$V_{5/3}$ and in $V_2$.\\

\noindent {\bf $V_2$ absorbing ball.} To show that the semigroup
$S(t): V_2 \rightarrow V_2$ has an absorbing ball in the phase space
$V_2=D(A)$ we take $H$ inner product of (\ref{a4}) with $Av(t)$:
\begin{equation}\label{a8}
\frac{d}{dt}\left[\|v(t)\|_1^2+\al^2 \|A v(t)\|^2\right]+2\nu \|A
v(t)\|^2 +2(B(v(t),v(t)),A v(t))=2(h,Av(t)).
\end{equation}
For the first term in the right hand side of \eqref{a8} we have
\begin{equation}\label{a9a}
|2(h,Av(t))|\leq  \frac1{\nu} \|h\|^2+ \nu \|Av(t)\|^2.
\end{equation}
By using the Agmon's inequality \eqref{Agm} and Young's inequality
\eqref{You} with $p= 4/3$ we can estimate the last term in the
left-hand side of (\ref{a8}) as follows
\begin{multline*}
2|(B(v,v),Av)|\leq C\|v\|_{L^{\infty}(\Om)}\|\|v\|_1\|\|Av\|\leq
C\|v\|_1^{3/2}\|\|Av\|^{3/2}\leq
\\
\frac34\epsilon\|Av\|^2 +\frac {C}{\epsilon^3}\|v\|_1^6.
\end{multline*}

\noindent Employing \eqref{a9a} and the  last inequality, with
$\epsilon=2\nu/3$, we obtain from (\ref{a8})
\begin{equation}\label{a9}
\frac{d}{dt}\left[\|v(t)\|_1^2+\al^2 \|A v(t)\|^2\right]+\nu \|A
v(t)\|^2\leq \frac1{\nu}\|h\|^2+ \frac{C}{\nu^3}\|v(t)\|_1^6.
\end{equation}
It follows from (\ref{a9}) that
$$
\frac{d}{dt}\left[\|v(t)\|_1^2+\al^2 \|A
v(t)\|^2\right]+d_0\left[\|v(t)\|_1^2+\al^2 \|A
v(t)\|^2\right]\leq \frac1{\nu}\|h\|^2+
\frac{C}{\nu^3}\|v(t)\|_1^6.
$$
Let $t_0$ be so that \eqref{a7} holds for all $t \geq t_0$. Then
integrating the last inequality over the interval $(t_0,t)$ we get
\begin{multline}\label{abV2}
\|v(t)\|_1^2+\al^2 \|A v(t)\|^2\leq
\\
\left[\|v(t_0)\|_1^2+\al^2 \|A v(t_0)\|^2\right] e^{-d_0(t-t_0)}+
\frac{R_2}{d_0}\left(1-e^{-d_0(t-t_0)}\right),
\end{multline}
where $R_2:=\frac1{\nu}\|h\|^2+
\frac{C}{\nu^3}M_1^6$.\\

\noindent The last inequality implies existence of an absorbing
ball

\begin{equation}\label{ball2}
 \mathcal{B}_2:=\{v \in V_2: \|Av\|\leq
M_2\},
\end{equation}
 where $M_2^2 =\frac{2R_2}{(\al^2+\la_1^{-1})d_0}.$ That is,
for all $t>>1,$ we have $\|Av(t)\|\leq M_2.$\\

\noindent Similarly, we can prove the following theorem
\begin{theorem}\label{atr2} If $h \in V_1$, then the semigroup $S(t): V_2\rightarrow V_2$
has a global attractor $\mathcal{A}_2\subset V_2$. The attractor
$\mathcal{A}_2 $ is compact, connected and invariant. Moreover,
$\mathcal{A}_2 $ is a bounded set in $V_3$.
\end{theorem}

\begin{remark}\label{A1A2} Let us note that in case
we assume  in Theorem \ref{atr1} that  $h\in V_1$, instead of  $h\in
H$, then the attractors $\mathcal{A}_1 $ and $\mathcal{A}_2 $
coincide.
\end{remark}

\section{Estimates for the Number of Determining Modes}

\noindent It is asserted, based on physical heuristic arguments,
that the long-time behavior of turbulent flows is determined by a
finite number degrees of freedom. This concept was formulated more
rigorously for 2D NS equations by introducing the notion of
determining modes in \cite{FP}. In \cite{FP} it was shown that there
exists a number $m$  such that if the first $m$ Fourier modes of two
different solutions of the NS equations have the same asymptotic
behavior, as $t\rightarrow \infty$, then the remaining
infinitely many number of modes have the same asymptotic behavior.\\

In \cite{La1} it was shown that the semigroup generated by the
initial boundary value problem for the 2D NS equations with
Dirichlet boundary condition has a global attractor which is
compact, invariant and connected . It was also established in
\cite{La1} that there exists a number $m$ such that if projections
of two different trajectories on the attractor on the $m$
dimensional subspace of $H$, spanned on the first $m$ eigenfunctions
of the Stokes operator, coincide for each $t\in\R$, then these
trajectories
completely coincide for each $t\in \R.$\\

The results obtained in \cite{FP} and \cite{La1} were developed,
generalized, and applied to various infinite dimensional dissipative
problems (see, e.g.,
\cite{Ch},\cite{CJTa},\cite{CJT},\cite{FMRT},\cite{FMTT},\cite{Foias-Titi},\cite{HT},\cite{ITi},
\cite{JT}, \cite{Jones-Titi},\cite{La2},\cite{OlTi},\cite{OlTi2} and
references therein).\\

In this section we are going to give estimates for the
number of  determining modes (both asymptotic and for trajectories
on the attractor) for 3D NSV
equations.\\

\noindent {\bf Asymptotic determining modes.} Let us denote by $P_m$
the $L^2$ -orthogonal projection from $H$ onto the $m$- dimensional
subspace
$H_m= \ \mbox{span} \ \{ w_1,w_2,...,w_m\}$. We set $Q_m=I-P_m$.\\
Let $v$ and $u$ be two solutions of NSV equations
\begin{equation}\label{d0}
v_t +\nu Av+\al^2 Av_t+B(v,v)=h(t), \ v(0)=v_0,
\end{equation}
\begin{equation}\label{d0a}
u_t +\nu Av+\al^2 Au_t+B(u,u)=g(t), \ v(0)=v_0.
\end{equation}
\begin{definition}\label{adm}
A set of modes $\{w_1,\cdots,w_m\}$ is called asymptotically
determining (see \cite{FMRT},\cite{FP}) if
$$
\lim_{t\rightarrow \infty} \|v(t)-u(t)\|_1=0
$$
whenever
$$
\lim_{t\rightarrow \infty} \|h(t)-g(t)\|_{-1}=0 \ \mbox{and} \
\lim_{t\rightarrow \infty} \|P_m(v(t)-u(t))\|_{1}=0.
$$
\end{definition}

\begin{theorem}\label{dm1} Assume that the following conditions are satisfied

\begin{equation}\label{e1}
\|h(t)\|_{-1}\leq{\mathbf h}<\infty, \ \forall t\in \R.
\end{equation}

\begin{equation}\label{e1a}
\lim_{t\rightarrow\infty}\|h(t)-g(t)\|_{-1}=0 \ \mbox{and} \
\lim_{t\rightarrow\infty}\|P_m(v(t)-u(t))\|=0.
\end{equation}
Then the first
 $m$ eigenfunctions of the Stokes operator are  asymptotically determining for the
 NSV equations with homogeneous Dirichlet boundary conditions, provided $m$ is
 large enough such that
\begin{equation}\label{asd}
\la_{m+1} >C\frac{\mathbf{h}^4}{\alpha^4\nu^8d_1^2}.
\end{equation}
\end{theorem}

\begin{proof} It is clear that the function $w=v-u$ satisfies
\begin{equation}\label{d1}
w_t +\nu Aw+\al^2 Aw_t+B(v,w)+B(w,v)-B(w,w)=\ta(t), \ v(0)=v_0,
\end{equation}
where $\ta(t)=h(t)-g(t)$.\\

\noindent It is clear from the proof of $(E_1)$ that
\begin{equation}\label{e2}
\limsup_{t\rightarrow \infty}\| v(t)\|_1 \leq \frac{{\mathbf
h}}{\alpha\nu\sqrt{d_1}}.
\end{equation}

\noindent  Multiplying (\ref{d1}) by $q(t)=Q_mw(t)$ in $H$ we obtain
\begin{multline}\label{d2}
\frac d{dt}\left[\|q\|^2+\al^2\| q\|_1^2\right]+2\nu\|
q\|_1^2+2b(q,v,q)=\\ 2(\ta,q)-2b(v,p,q)-2b(p,p,q)+2b(q,p,q),
\end{multline}
where $p=P_m w$.\\
\noindent Before estimating the terms of \eqref{d2} we observe
that for each $\phi\in V$ we have
\begin{equation}\label{Pom}
\|Q_m \phi \|_1\geq \la_{m+1}\|Q_m\phi\| \ \mbox{and}  \ \|P_m
\phi\|_1\leq \la_m \|P_m \phi\|.
\end{equation}

\noindent Due to the  inequality \eqref{bf1} the term $b(q,v,q)$ has
the following estimate:
\begin{equation}\label{bfu1}
2|b(q,v,q)|\leq C\|q\|^{1/2}\|\|q\|_1^{3/2}\|v\|_1\leq
\frac{C}{\la_{m+1}^{1/4}} \|q\|_1^{2}\|v\|_1.
\end{equation}

\noindent The first term in the right-hand side of \eqref{d2} has
the estimate
\begin{equation}\label{ft}
2|(\theta,q)|\leq \frac 2{\nu}\|\theta\|_{-1}^2
+\frac{\nu}2\|q\|_1^2.
\end{equation}

Employing the inequalities \eqref{bf1ad} and \eqref{Pom}  we
estimate the second term in the right-hand side of (\ref{d2}) as
follows
\begin{equation}\label{d6}
2|b(v,p,q)|\leq \|v\|_1\|p\|_1\|q\|^{1/2}\|q\|_1^{1/2}\leq
C\lambda_m
\lambda_{m+1}^{-1/4}\|p\|_1\left(\|q\|_1^2+\|v\|_1^2\right).
\end{equation}
Other terms in the right-hand side of (\ref{d2}) can be estimated in
a similar way to (\ref{d6}). \\
Using estimates \eqref{bfu1}-\eqref{d6} and the estimates of other
terms in the right-hand side of  (\ref{d2}) we obtain
\begin{equation}\label{ns}
\frac d{dt}\left[\|q\|^2+\al^2\|q\|_1^2\right]+\frac{\nu}2\|q\|_1^2+
\|q\|_1^2\left(\nu-\frac{C}{\la_{m+1}^{1/4}}\| v\|_1\right)\leq
b(t),
\end{equation}
where $b(t)$  is satisfying the corresponding condition of Lemma
1.\\

\noindent Let us choose $t_1>0$ so large that $\|v(t)\|_1\leq M_1,$
for all $ t\geq t_1$ and  $m$ so that $
\mu(m):=\la_{m+1}-(\frac{CM_1}{\nu})^4>0$. Then it follows from the
last inequality the following relation
$$
\frac
d{dt}\left[\|q\|^2+\al^2\|q\|_1^2\right]+\frac{\nu}2\|q\|_1^2\leq
b(t),\ \mbox{for all} \  t\geq t_1,
$$
or
\begin{equation}\label{d7}
\frac d{dt}\left[\|q\|^2+\al^2\|
q\|_1^2\right]+d_m\left[\|q\|^2+\al^2\|q\|_1^2\right]\leq b(t), \ \
\mbox{for all} \ t\geq t_1,
\end{equation}
where $d_m=\frac{\nu}4\min\{\frac{1}{\al^2},\la_{m+1}\}.$\\

\noindent Thus, due to Lemma 1 the statement of the theorem follows.
\end{proof}
\begin{remark}\label{3DNS} Let us observe that the number
$m$, for which
$\lambda_{m+1}>\frac{C\mathbf{h}^4}{\nu^8\lambda_1^2}$ holds, is
 an upper bound for the minimal number of asymptotically determining
modes for weak solutions (i.e., solutions belonging to
$L^{\infty}(\R^+;H)\cap L_{loc}(\R^+;V)$) of the initial boundary
value problem for
the 3D Navier Stokes equations.\\
In fact,  for  weak solutions of NS equations instead of
\eqref{ns} we have
$$
\frac d{dt}\|q\|^2+ \la_{m+1}^{3/4}\left(\nu\la_{m+1}^{1/4}-C\|
v\|_1\right)\|q\|^2\leq b(t),
$$
and instead of \eqref{e2}  we have for weak solutions of NS
equations (see, e.g., \cite{CF}, \cite{Constantin-Doering-Titi},
\cite{HT} and \cite{Tem})
$$
\limsup_{t\rightarrow \infty}\frac1T\int_t^{t+T}\|
v(\tau)\|_1^2d\tau\leq \frac{\mathbf{h}^2}{T\nu^3\la_1^2}+
\frac{\mathbf{h}^2}{\nu^2\la_1}.
$$
Hence
$$
\limsup_{t\rightarrow \infty}\frac1T\int_t^{t+T}\|
v(\tau)\|_1d\tau\leq \frac{\mathbf{h}}{\sqrt{T}\nu^{3/2}\la_1}+
\frac{\mathbf{h}}{\nu\sqrt{\la_1}}.
$$
Thus, the function $a(t):=
\la_{m+1}^{3/4}\left(\nu\la_{m+1}^{1/4}-C\| v\|_1\right)$ satisfies
conditions of Lemma \ref{le1} provided $T$ is large enough and
$$
\la_{m+1}> C\frac{\mathbf{h}^4}{\nu^8\la_1^2}.
$$

\noindent Different estimates of asymptotic determining modes for
weak solutions of 3D NS equations are obtained in \cite{CFMT} (see
also \cite{CFT}, \cite{FMRT} and references therein). The estimate
obtained in \cite{CFMT} involves generalization of the so called
mean rate dissipation of energy, per mass and time, i.e. it involves
$$
\varepsilon=\nu \limsup_{t\rightarrow\infty}\frac1t\int_0^t\sup_{x
\in \Om}\|\Nx v(x,\tau)\|^2d\tau.
$$
For other related results concerning estimates of the number of
asymptotic determining degrees of freedom for weak solutions of the
3D NS equations see, e.g., \cite{Constantin-Doering-Titi}, \cite{HT}
and references therein.
\end{remark}

\noindent {\bf Determining modes on the attractor.} Next we give an
estimate of determining modes for trajectories on the attractor.

\begin{definition}\label{adm}
A set of modes $\{w_1,\cdots,w_m\}$ is called determining on the
attractor (in the sense of  \cite{La1}) if for each two trajectories
$v(t)$ and $u(t)$ on the attractor $\mathcal{A}_1$ the equality
$$
 \|P_m(v(t)-u(t))\|_{1}=0, \ \ \mbox{for all} \  t \in \R
$$
implies
$$
v(t)=u(t), \ \ \forall t\in R.
$$
\end{definition}

\noindent Let $v$ and $u$ be arbitrary two trajectories in the
attractor $\mathcal{A}_1$ of (3.1). Then $w=v-u$ satisfies
\begin{equation}\label{dL1}
w_t+\al^2 Aw_t+\nu Aw +B(w,v)+B(u,w)=0.
\end{equation}
Taking the inner product of (\ref{dL1}) with $q=Q_m w$ we get
\begin{equation}\label{dL2}
\frac d{dt}\left[ \|q\|^2+\|q\|_1^2\right]+2\nu\|
q\|_1^2=-2b(w,v,q)-2b(u,w,q).
\end{equation}
Assume that $P_mw(t)=0$, for all $t \in \R$, then $Q_mw=q$ satisfies
\begin{equation}\label{dL2a}
\frac d{dt}\left[ \|q\|^2+\|q\|_1^2\right]+2\nu\| q\|_1^2=2b(q,v,q).
\end{equation}
Due to  \eqref{bf1}  we have
$$
|2b(q,v,q)| \leq C \|q\|^{1/2}_1\|q\|_1^{3/2}\|v\|_1
$$
Noting that on the attractor $\mathcal{A}_1$ we have $\|v\|_1 \leq
M_1$, we employ the last inequality, and inequality \eqref{Pom} to
obtain from (\ref{dL2})
\begin{equation}\label{dL3}
\frac d{dt}\left[\|q\|^2+\al^2\|q\|_1^2\right]+\nu\|q\|_1^2 +
\|q\|^{1/2}\|q\|^{3/2}_1\left(\nu \la_{m+1}^{1/4}-CM_1\right) \leq
0.
\end{equation}
Let us choose $m$, large enough, so that $\la_{m+1}\geq
(\frac{M_1C}{\nu})^4$. Then (\ref{dL3}) implies
$$
\frac d{dt}\left[ \|q\|^2+\alpha^2\|q\|_1^2\right]+l_m\left[
\|q\|^2+\alpha^2\| q\|_1^2\right]\leq 0,
$$
where $l_m=\frac{\nu}2\min\{\la_{m+1},\frac1{\al^2}\}$.\\

\noindent Finally, we integrate the last inequality and get
\begin{equation}\label{dL4}
\|q(t)\|^2+\al^2\|q(t)\|_1^2\leq \exp[-l_m(t-s)]\left[
\|q(s)\|^2+\al^2 \|q(s)\|_1^2\right].
\end{equation}
Passing to the limit as $s\rightarrow -\infty$ we obtain
$$
\|q(t)\|^2+\al^2\|q(t)\|_1^2=0, \ \mbox{for all} \  t\in \R.
$$
Thus, the following theorem is true.
\begin{theorem}\label{dmatr} Let  $v$ and $u$ be two  solutions of the
problem (\ref{a1})-(\ref{a3}) from the attractor $\mathcal{A}_1$.
Assume that $P_m(u(t))= P_m(v(t)) , \forall t \in \R$, where $m$ is
so that
\begin{equation}\label{ldm1}
\la_{m+1}\geq  C\frac{\|h\|_{-1}^4}{\alpha^4\nu^8d_1^2}.
\end{equation}
Then $v(t)=u(t), \ \mbox{for all} t \in \R.$
\end{theorem}

\section{Estimates of Dimensions of the  Global Attractor}

In this section we show the differentiability of the semigroup with
respect to the initial data. This is to prepare for implementing
Theorem \ref{frdim} in order to estimate the dimension of the global
attractor.

\begin{theorem}\label{dat1}
Let $u_0$ and $v_0$ be two elements of $V$. Then there is a constant
$K= K(\|u_0\|_1,\|v_0\|_1)$ such that
\begin{equation}\label{da1}
\|S(t)v_0-S(t)u_0 -\Lambda(t)(v_0-u_0)\|_1\leq K\|v_0-u_0\|_1^2,
\end{equation}
where  the linear operator $\Lambda(t): V\rightarrow V,$ for $ t>0$
is the solution operator of the problem
\begin{equation}\label{da2}
\xi_t+\al^2 A\xi_t+A\xi+B(\xi,v)+B(v,\xi)-B(\xi,\xi)=0, \ \xi(0)=
v_0-u_0,
\end{equation}
and $v(t)=S(t)v_0$. That is, for every $t>0$, the map $S(t)v_0$, as
a map $S(t):V \rightarrow
 V$ is Fr$\acute{e}$chet differentiable with respect to the initial data, and
its
 Fr$\acute{e}$chet derivative $D_{v_0}(S(t)v_0)w_0=\Lambda(t)w_0.$
\end{theorem}

\begin{proof} It is easy to see that the function
$\eta(t):=v(t)-u(t)-\xi(t)=S(t)(v_0-u_0)-\xi(t)$ satisfies
$$
\eta_t+\al^2 A\eta_t+\nu A\eta+B(\eta,v)+B(v,\eta)-B(w,w)=0,
$$
where $w=v-u.$ Taking the inner product of the last equation with
$\eta$ we obtain
\begin{equation}\label{da3}
\frac{d}{dt}\left[\|\eta\|^2+\al^2\|\eta\|_1^2\right]+2\nu\|\eta\|_1^2=-2b(\eta,v,\eta)
-2b(w,w,\eta).
\end{equation}
By using inequalities (\ref{a7}) and (\ref{Lad4}) and Young's
inequality we can estimate the terms in the right-hand side of
(\ref{da3}) as follows.\\
\noindent By \eqref{bf1} we have
$$
|2b(\eta,v,\eta)|\leq C\|v\|_1\|\eta\|^{1/2}\|\eta\|_1^{3/2}\leq
CM_1\|\|\eta\|^{1/2}\| \eta\|_1^{3/2}\leq
\frac{CM_1}4(\|\eta\|^2+3\|\eta\|_1^2).
$$
By \eqref{bf4}
$$
|2b(w,w,\eta)|=|2b(w,\eta,w)|\leq
C\la_1^{-\frac12}\|\eta\|_1\|w\|_1^2\leq \nu\|\eta\|_1^2+
\frac{C}{4\nu\la_1}\|w\|_1^4.
$$
Hence, we obtain from (\ref{da3})
\begin{equation}\label{da4}
\frac{d}{dt}\left[\|\eta\|^2+\al^2\|\eta\|_1^2\right]\leq
\frac{CM_1}4(\|\eta\|^2+3\|
\eta\|_1^2)+\frac{C}{4\nu\la_1}\|w\|_1^4.
\end{equation}
The function $w(t)=v(t)-u(t)=S(t)v_0-S(t)u_0$ satisfies
$$
w_t+\al^2 Aw_t+\nu Aw+B(w,v)+B(v,w)-B(w,w)=0, \ w(0)= v_0-u_0:=w_0.
$$
Taking the inner product of the last equation with $w$,  and using
\eqref{bf2} and  $(E_2)$ we obtain
\begin{multline*}
\frac d{dt}\left[\|w\|^2+\al^2\| w\|_1^2\right]+2\nu
\|w\|_1^2=2b(w,v,w)\leq 2C\lambda_1^{1/4}\|v\|_1\|w\|_1^{2} \leq\\
\kappa_1\|v\|_1\left[\|w\|^2+\al^2\| w\|_1^2\right],
\end{multline*}
where
$k_1=2C\lambda_1^{1/4}\alpha^{-4}\left[\|v(0)\|^2+\alpha^2\|v(0)\|_1^2+(1/\nu
d_0)\|h\|_{-1}^2\right]^{1/2}$. Integrating last inequality  we get
\begin{equation}\label{da5}
\|w(t)\|_1^2\leq (1+\frac1{\la_1\al^2})\| w(0)\|_1^2\exp(\kappa_1t).
\end{equation}
It follows from (\ref{da4}) and (\ref{da5}) that
$$
\frac{d}{dt}\left[\|\eta\|^2+\al^2\|\eta\|_1^2\right]\leq A_1
\left[\|\eta\|^2+\al^2\|\eta\|_1^2\right]+A_2\|
w(0)\|_1^4\exp(2\kappa_1t).
$$
Integrating and using Gronwall's inequality :
$$
\|\eta(t)\|_1^2\leq A(t)\|w(0)\|_1^4,
$$
where $A(t):=\frac{A_2}{2\kappa_1\al^2}\exp(2\kappa_1+A_1)t.$ So we
have
$$
\frac{\|v(t)-u(t)-\xi(t)\|_1}{\|v_0-u_0\|_1}\leq
\sqrt{A(t)}\|v_0-u_0\|_1.
$$
Thus the differentiability of $S(t)$ with
respect to the initial data  follows.\\
\noindent We rewrite (\ref{a4}) in the following form
\begin{equation}\label{da6}
\hat{v}_t=-\frac{\nu}{\al^2}\hat{v}+\frac{\nu}{\al^2}G^{-2}\hat{v}-G^{-1}B(G^{-1}\hat{v},G^{-1}\hat{v})+G^{-1}h,
\end{equation}
where $G^2 =I+\al^2 A,$ and $ \hat{v}=Gv$.\\
The equation of linear variations corresponding to (\ref{da6}) has
the form
\begin{equation}\label{da7}
w_t= L(t)w,
\end{equation}
where
$$
L(t)w:=-
\frac{\nu}{\al^2}w+\frac{\nu}{\al^2}G^{-2}w-G^{-1}B(G^{-1}w,
G^{-1}\hat{v})-G^{-1}B(G^{-1}\hat{v},G^{-1}w).
$$
Now we consider the quadratic form
$$
(L(t)w,w)=- \frac{\nu}{\al^2}\|w\|^2 +
\frac{\nu}{\al^2}\|G^{-1}w\|^2 -b(G^{-1}w,G^{-1}\hat{v},G^{-1}w).
$$
By using inequality \eqref{bf1} and the inequality
$\|G^{-1}u\|_1\leq \frac{1}{\alpha}\|u\|$ we get
$$
|b(G^{-1}w,G^{-1}\hat{v},G^{-1}w)|\leq
\frac{1}{\alpha^{5/2}}\|G^{-1}w\|^{1/2}\|w\|^{3/2}\|\hat{v}\|.
$$

\noindent Employing Young's inequality with $p=4/3,
\epsilon=2\nu/(3\alpha^2)$, and the fact that on the global
attractor $\mathcal{A}_1$ the estimate $\|\hat{v}\|\leq
(\lambda_1+\alpha^2)^{1/2} M_1$ holds, we obtain
$$
|b(G^{-1}w,G^{-1}\hat{v},G^{-1}w)|\leq \frac{\nu}{2\al^2}\|w\|^2
+\frac{C(\lambda_1+\alpha^2)^2M_1^4}{\nu^3\alpha^4}\|G^{-1}w\|^2.
$$
Due to the last inequality the quadratic form $(L(t)w,w)$ has the
following estimate
\begin{equation}\label{da8}
(L(t)w,w)\leq - \frac{\nu}{2\al^2}\|w\|^2
+\left(\frac{\nu}{\al^2}+\frac{C(\lambda_1+\alpha^2)^2M_1^4}{\nu^3\alpha^4}\right)\|G^{-1}w\|^2
\end{equation}
\noindent Thus, we can use Theorem \ref{frdim} to get the desired
estimate for the  fractal dimension of the attractor $\mathcal{A}_1$
\begin{equation}\label{da9}
d_f(\mathcal{A}_1)\leq
C\frac{(\lambda_1+\alpha^2)^2M_1^4}{\nu^4\alpha^2}+2\leq
C\frac{(\lambda_1+\alpha^2)^2\|h\|_{-1}^4}{\nu^8\alpha^6d_1^2}+2.
\end{equation}
We recall that $ M_1=\frac{2}{\nu\al\sqrt{ d_1}}\|h\|_{-1}, \ \ d_1=
\frac12\min\{\alpha^{-2},\lambda_1\}.$ \noindent Let us note that in
our situation
$$\bar{h}_0(t)= \frac{\nu}{2\al^2}, s_0=0, {s_1}=-1,
\bar{h}_{s_1}(t)=\frac{\nu}{\al^2}+\frac{C(\lambda_1+\alpha^2)^2M_1^4}{\nu^3\alpha^4}
$$
and $\bar{h}_{s_k}(t)= 0, k\geq 2$.\\
\end{proof}

\section{ The Inviscid Limit}

\noindent Here  we  show that when $\nu \rightarrow 0 $ the weak
solution of the initial boundary value problem for the NSV system,
i.e. of the problem \eqref{a1}-\eqref{a3}, is tending to the weak
solution of the initial boundary value problem for the inviscid
simplified Bardina model
\begin{equation}\label{Ba1}
\begin{array}{c}
 u_t-\al^2\Dx u_t+(u\cdot \Nx )u+\Nx p=f,  \ \ x \in \Om, t>0, \\ \Nx \cdot  u
 =0, \ \  x \in \Om, t>0,
\end{array}
\end{equation}
\begin{equation}\label{Ba2}
 u(x,0)=v_0(x),\  x \in \Om; \ \ u(x,t)= 0, \ x \in \partial \Om,
t>0.
\end{equation}

The problem of existence and uniqueness of solutions of the initial
boundary value problem, with periodic boundary conditions, for  the
3D viscous and inviscid simplified Bardina models is studied in
\cite{CLT}. In particular, it is shown in \cite{CLT} that the
problem \eqref{Ba1}-\eqref{Ba2} has a unique solution $u\in
C^1(\R;V)$, for  initial value $u_0\in V$.\\
\noindent Applying to \eqref{Ba1} the Helmholtz-Leray operator $P$
we obtain the equivalent functional differential equation
\begin{equation}\label{Ba3}
 u_t+ \al^2Au_t +B(u,u)=h,
\end{equation}
\begin{equation}\label{Ba4}
u(0)=v_0.
\end{equation}
Let $v(t)$ be the solution of \eqref{Ba1} with initial $v(0)=v_0\in
V$. Denote by $w=v-u$. Then $w$ satisfies the relation
\begin{equation}\label{Ba4}
 w_t+ \al^2Aw_t +B(w,v)+B(u,w)=-\nu Av,
\end{equation}
\begin{equation}\label{Ba5}
w(0)=0,
\end{equation}
which holds in the space $V'$. Taking the action of \eqref{Ba4} on
$w$, which belongs to $V$, and using a Lemma of Lions-Magenes
concerning the derivative of functions with values in Banach space
(cf. Lemma 1.2 Chap. III-p.169-\cite{TT84}), we obtain
\begin{equation}\label{Ba6}
\frac d{dt}\left[\|w\|^2+\al^2\|w\|_1^2\right] =-2\nu (\Nx v, \Nx
w)- 2b(w,v,w).
\end{equation}
For the first term in the right-hand side we have
$$
|2\nu (\Nx v, \Nx w)|\leq \nu^2\|v\|_1^2+\| w\|_1^2.
$$
The second term we estimate by using the  inequality \eqref{bf2}
$$
|2b(w,v,w)|\leq C\la_1^{1/4}\|v\|_1\|w\|_1^2.
$$
Utilizing last two inequalities in \eqref{Ba6} we get
\begin{multline*}
\frac d{dt}\left[\|w\|^2+\al^2\|w\|_1^2\right] \leq
\nu^2\|v\|_1^2+ \left(1+C\la_1^{1/4}\|v\|_1\right)\|w\|_1^2\leq\\
\nu^2\|v\|_1^2+
\al^{-2}\left(1+2C\la_1^{1/4}\|v\|_1\right)\left[\|w\|^2+\al^2\|w\|_1^2\right].
\end{multline*}
Integrating the last inequality and using the standard Gronwall's
lemma we get the estimate
\begin{multline}\label{Ba7}
\|w(t)\|^2+\al^2\|w(t)\|_1^2 \leq \nu^2 \int_0^t\|v(\tau)\|_1^2d\tau
\exp\left(\frac
t{\al^2}+\frac{2C\la_1^{1/4}}{\al^2}\int_0^t\|v(\tau)\|_1d\tau\right).
\end{multline}
Next we show that on each finite interval $[0,T]$ we can estimate
$\|v\|_1$ by a constant depending only on $\|v_0\|,\|v_0\|_1$ and
the parameter $\alpha$. Indeed, \eqref{a6} implies that
$$
\frac{d}{dt}\left[\|v(t)\|^2+ \al^2\|v(t)\|_1^2\right]\leq
\alpha^{-2}\|h\|_{-1}^2+ \alpha^2\|v(t)\|_1^2.
$$
Integrating the last inequality over $(0,t)$ with respect to time
variable we obtain
\begin{multline*}
\|v(t)\|^2+\alpha^2\|v(t)\|^2_1\leq \|v_0\|^2 +\alpha^2\| v_0\|^2_1+
t\alpha^{-2}\|h\|_{-1}^2+ \int_0^t\left[\|v(\tau)\|^2+\alpha^2\|
v(\tau)\|^2_1\right]d\tau.
\end{multline*}
By using the Gronwall inequality we get
$$
\|v(t)\|^2_1\leq D_Te^T, \, \mbox{for all} \, t \in [0,T].
$$
Here $D_T:= \alpha^{-2}\left[\|v_0\|^2 +\alpha^2\| v_0\|^2_1+
T\alpha^{-2}\|h\|_{-1}^2\right].$\\

\noindent  Hence \eqref{Ba7} implies

\begin{equation}\label{Ba7a}
\|w(t)\|^2+\al^2\|w(t)\|_1^2 \leq  \nu^2 TD_Te^T
\exp\left(\alpha^{-2}T+2C\alpha^{-2}\lambda^{1/4}TD_T^{1/2}e^{T/2}\right).
\end{equation}

\begin{remark}\label{nu} The problem of convergence of solutions of
the NSV equations to solutions of NS equations as $\alpha
\rightarrow 0$ was studied in \cite{Osk2}. It is shown in
\cite{Osk2} that strong solutions of the NSV equations converge to
strong solutions of the NS equations as $\alpha \rightarrow 0$,
under specified smallness conditions on the initial data of the
problem.
\end{remark}

\begin{remark}\label{bc} The results obtained in this paper are
valid also for the solutions of the initial boundary value problem
for the 3D NSV equations with periodic boundary conditions.
\end{remark}

\noindent Finally we would like to notice that the results reported
here can be extended to other similar equations, a subject of future
work. For instance, for the 3D equations of motion of Kelvin-Voight
fluids
of order $L\geq 1$:\\
$$
v_t+(v\cdot \Nx)v-\mu_0\Dx v_t -\mu_1\Dx v-\sum_{l=1}^L\beta_l\Dx
u_l+\Nx p=f,
$$
$$
\partial_t u_l+\alpha_lu_l-v=0, \ \
l=1,...,L.
$$
\noindent where $\mu_0,\mu_1, \beta_l, \alpha_l>0, l=1,...,L.$ Also
for the generalized Benajamin-Bona-Mahony (GBBM)
equation:\\
\begin{equation}\label{bbm1}
u_t- \alpha^2\Dx u_t + \nu \Dx u + \Nx \cdot \vec{F}(u)=h,
\end{equation}
where  a smooth vector field $\vec{F}(u)$ satisfies  the   growth
condition
$$
|\vec{F}(u)|\leq C(1 + |u|^2).
$$
\noindent The problem of existence of a finite dimensional global
attractor and estimates for the number of determining modes on the
global attractor of  Kelvin-Voight fluids of order $L\geq 1$ is
established in \cite{KKO}. In \cite{WY} the existence of a finite
dimensional  global attractor is established for 1D GBBM equation
under periodic boundary conditions. Existence of a finite
dimensional global attractor for 3D GBBM under periodic boundary
conditions is proved in \cite{CKP}. In \cite{SSW} it was shown the
existence of the global attractor for GBBM equation in $H^1(\R^3)$.
Moreover, the existence of a global attractor for a similar
two-dimensional model describing the
motion of a second-grade fluid is established in \cite{MRW}.\\

\noindent {\bf Acknowledgements.}  The work of the V.~K.~Kalantarov
was supported in parts by The Scientific and Research Council of
Turkey, grant no. 106T337. The work of E.~S.~Titi was supported in
parts by the NSF grant no.~DMS-0504619, the ISF grant no.~120/6, and
the BSF grant no.~2004271.

\end{document}